\documentclass[twocolumn,amsthm]{autart}    

\usepackage{graphicx}          

\usepackage{amssymb}
\usepackage{amsthm}
\newtheorem{theorem}{Theorem}

\newtheorem{lemma}{Lemma}
\newtheorem{proposition}{Proposition}

\newtheorem{remark}{Remark}
\newtheorem{definition}{Definition}
\newtheorem{corollary}{Corollary}
\usepackage{mathrsfs}
\usepackage{amsmath}
\usepackage{enumerate}
\usepackage{accents}
\usepackage{marvosym}
\usepackage{wasysym}

\usepackage{subfiles}

\begin{document}

\begin{frontmatter}

\title{On the mixed monotonicity of polynomial functions\thanksref{footnoteinfo}} 

\thanks[footnoteinfo]{This paper was not presented at any IFAC 
meeting. Corresponding author A.M.Tahir. }

\author[UW]{Adam M Tahir}\ead{aerotahir@gmail.com} 

\address[UW]{William E. Boeing Department of Aeronautics and Astronautics, University of Washington, Seattle, WA, USA}  


\begin{abstract}                          
In this paper, it is shown that every polynomial function is mixed monotone globally with a polynomial decomposition function. For univariate polynomials, the decomposition functions can be constructed from the Gram matrix representation of polynomial functions. The tightness of polynomial decomposition functions is discussed. Several examples are provided. An example is provided to show that polynomial decomposition functions, in addition to being global decomposition functions, can be much tighter than local decomposition functions constructed using local Jacobian bounds. Furthermore, an example is provided to demonstrate the application to reachable set over-approximation. 
\end{abstract}

\end{frontmatter}

\section{Introduction}

Mixed monotonicity can be seen as a generalization of monotonicity. A function is mixed monotone if it can be decomposed into a function with two arguments that is monotonically increasing in the first argument and monotonically decreasing in the second argument. The two argument function is referred to as the mixed monotone decomposition function.   In general, it is difficult to compute the decomposition function of a function to show that it is mixed monotone. To compute mixed monotone decomposition functions, (local or global) properties of classes of functions are exploited. For example, (local or global) Jacobian bounds are frequently exploited to compute mixed monotone decomposition functions for differentiable functions.  

There are many important applications of mixed monotonicity. For example, a major application is in efficient computation of over-approximations of reachable sets by intervals \cite{CooganFinitieAbstraction,9304391}. The efficiency of computation comes from the fact that the over-approximation of a reachable set can be achieved by propagating two points (the upper and lower bounds of the interval containing the initial condition) forward in time with dynamics that include mixed monotone decompositions which introduce coupling between the dynamics of the lower bound and the dynamics of the upper bound. The coupling ensures that the upper bound always remains an upper bound and the lower bound always remains a lower bound. Hence, an over-approximation of the reachable set with an interval follows. 


Another important application of mixed monotonicity is in interval observer synthesis. Interval observers are set-based state estimators which provide a bounded interval to which the systems state is guaranteed to belong to. Interval observers are designed to ensure monotonicity of the upper and lower bounds of the interval  (i.e. the upper bound always remains an upper bound and the lower bound always remains a lower bound) in addition to stability of the bounds. Mixed monotonicity is exploited in interval observer synthesis in a similar way to how it is used in reachable set prediction \cite{IntervalLip,IntervalEstimationNonlin,Me,LCSS_paper,khaj}.

The decomposition function is not unique, which has consequences for the aforementioned applications. It is desirable to find a so-called tight (in a sense to be made more precise later on) decomposition function so that the interval estimate in an interval observer is as small as possible. Moreover, it is desirable to compute the smallest over-approximation of the reachable set, which can also be accomplished by finding a tight decomposition function.

Polynomial functions occur frequently in control theory. In this paper, the Gram matrix representation of polynomials, where a polynomial is expressed in a quadratic form, is exploited to compute mixed monotone decompositions of univariate polynomial functions. Using the Gram matrix representation naturally leads to semidefinite programming problems \cite{Parrilo:2003aa,PolynomialLMIs}. In constructing a mixed monotone decomposition, the semidefinite program is related to the tightness of the decomposition function. 

The contribution of this paper is that it is shown that every polynomial function is mixed monotone globally with a polynomial decomposition function.  Moreover, a procedure to construct the polynomial decomposition function is presented. Examples are provided which show the tightness of the derived decomposition functions and show an application to reachable set over-approximation.

\section{Preliminaries}
\subsection{Notation}
The set of real numbers is denoted by  $\mathbb{R}$, $\mathbb{R}^n$ denotes the set of $n$-dimensional real vectors, $\mathbb{R}^{n\times m}$ denotes the set of real $n\times m$ matrices, and $\mathbb{N}_0$ denotes the set of natural numbers including 0. The $i$th element of the canonical basis for $\mathbb{R}^n$ is denoted by $e_i$.

The transpose of a vector $x\in\mathbb{R}^n$ is denoted by $x^\top$. The set of  symmetric matrices of dimension $n\times n$ is denoted by $\mathbb{S}^n$. A matrix $M\in\mathbb{S}^n$ is positive (resp. negative) semidefinite if $x^\top Mx\ge 0$ (resp. $\le 0$) for all $x\in\mathbb{R}^n$. The notation $M\succeq 0$ (resp. $\preceq 0$) denotes that $M$ is positive (resp. negative) semidefinite. For a vector $x\in\mathbb{R}^n$, $\|x\|$ denotes the Euclidean norm of $x$. 

For a square matrix $M\in\mathbb{R}^{n\times n}$, the Frobenius norm of $M$ is denoted by $\|M\|_F$ and the 1-matrix norm is denoted by $\|M\|_1$.  

The derivative of $f(x):\mathbb{R}\to\mathbb{R}$ is denoted by $f'(x)$ and the indefinite integral of $f(x)$ is denoted by $\int f(x)dx$. 

An interval with lower bound $a$ and upper bound $b$ is denoted by $[a,b]$. 
\subsection{Polynomials and the Gram matrix representation}

\begin{definition}[Definition 2.2 in \cite{Parrilo:2003aa}]
A polynomial $p$ in $x\in\mathbb{R}^n$ is a finite linear combination of monomials:
\begin{align*}
p(x) = \sum_\alpha c_\alpha x^\alpha =\sum_\alpha c_\alpha x_1^{\alpha_1}\dots x_n^{\alpha_n}, \;\; c_\alpha\in\mathbb{R},
\end{align*}
where the sum is over a finite number of $n$-tuples $\alpha = (\alpha,\dots, \alpha_n)$, $\alpha_i\in\mathbb{N}_0$. A polynomial is called a univariate polynomial if $n=1$. Otherwise, it is called a multivariate polynomial.
\end{definition}
 The total degree of the monomial $x^\alpha$ is $\sum_{i=1}^n\alpha_i$. The total degree of a polynomial is the highest degree of its monomials.

A univariate polynomial $p$ of degree $d$ can be represented in a quadratic form\footnote{Note that multivariate plynomials can also be expressed using the Gram matrix representation, but in this paper the Gram matrix representation is exploited only in the univariate polynomial case.} as follows \cite{Parrilo:2003aa,PolynomialLMIs}:
\begin{align*}
p(x) = {x^{\{\sigma\}}}^\top(G+L(\alpha)){x^{\{\sigma\}}},
\end{align*}
where  $G\in\mathbb{S}^\sigma$, $x^{\{\sigma\}}$ is a vector of monomials of $x$ up to degree $\sigma$, that is, 
\begin{align*}
x^{\{\sigma\}} = \begin{bmatrix}1 & x & \dots & x^{\sigma-1} & x^\sigma\end{bmatrix}^\top,
\end{align*}
the degree $\sigma$ is related to the degree $d$ by
\begin{align*}
\sigma = \left\{\begin{array}{lr}
\frac{d}{2}, & d \text{ is even},\\
\frac{d+1}{2}, & d \text{ is odd},
\end{array}\right.
\end{align*}
and $L(\alpha)$ is a linear parameterization of the set 
\begin{align}
\mathcal{L} = \left\{L\in\mathbb{S}^{\sigma}:  {x^{\{\sigma\}}}^\top L{x^{\{\sigma\}}}=0, \forall x\in\mathbb{R}\right\}.\label{e:Lset}
\end{align}
The matrix $G+L(\alpha)$ is referred to as the Gram matrix for the polynomial $p(x)$. 

\subsection{Linear algebra}
The following lemma which guarantees the existence of a decomposition of any symmetric matrix into a difference of  two a positive semidefinite matrices will be used in proving the main result of this paper. 
\begin{lemma}\label{emma}
Let $A\in\mathbb{S}^n$. There exist matrices $U,V\succeq 0$ such that $A=U-V$.
\end{lemma}
\begin{proof}
The matrix $A$ can be decomposed as $A = \sum_{i=0}^n \lambda_i u_iu_i^\top,$ where $\lambda_i$ is the $i$th eigenvalue of $A$, and $u_i$ is the corresponding eigenvector (note that since $A$ is symmetric, all of its eigenvalues and eigenvectors are real). By taking $U = \sum_{i=1}^n \max(\lambda_i,0) u_iu_i^\top$ and $V = -\sum_{i=1}^n \min(\lambda_i,0) u_iu_i^\top$ the result follows.
\end{proof}
\begin{remark}\label{r:uniqueness}
Note that the decomposition of a symmetric matrix into a difference of  two a positive semidefinite matrices is not unique. This can easily be seen from the fact that if $A=U-V$ where $U,V\succeq 0$ , then $A=\tilde{U}-\tilde{V}$ where $\tilde{U}=U+R\succeq 0$ and $\tilde{V}=V+R\succeq 0$ for any $R\succeq 0$ of appropriate dimension. 
\end{remark}

\section{Background on Mixed Monotonicity}
A mixed monotone function is formally defined as follows:
\begin{definition}[ cf. Definition 4 in  \cite{SuffMixMonotone}] \label{d:mm}
$f:\mathcal{X}\to\mathcal{T}$ is {\it mixed monotone on $\mathcal{X}$}  if $\exists g:\mathcal{X}\times \mathcal{X}\to \mathcal{T}$  that satisfies the following:
\begin{enumerate}
\item $f$ is ``embedded'' on the diagonal of $g$, i.e. $g(x,x)=f(x)$ for all $x\in\mathcal{X}$;
\item $g$ is monotonically increasing in its first argument, i.e. $x_1\ge x_2\implies g(x_1,y)\ge g(x_2,y)$ for all $x_1,x_2,y\in\mathcal{X}$; and
\item  $g$ is monotonically decreasing in its second argument, i.e. $y_1\ge y_2\implies g(x,y_2)\ge g(x,y_1)$ for all $x,y_1,y_2\in\mathcal{X}$. 
\end{enumerate}
The function $g$ is referred to as the {\it mixed monotone decomposition function for $f$ on $\mathcal{X}$}.
\end{definition}

Note that the above definition works in the cases where $x$ is either univariate or  multivariate and $f$ is either scalar or vector. In the multivariate and vector cases, the inequality signs are element-wise inequalities. 

If a function $f(x)$ is mixed monotone on $\mathcal{X}$ with a decomposition function $g(x,y)$, then 
\begin{align}
g(\underline{x},\overline{x})\le f(x)=g(x,x)\le g(\overline{x},\underline{x})\label{e:boundingproperty}
\end{align}
for all  $\underline{x},x,\overline{x}\in\mathcal{X}$ such that $\underline{x}\le x\le \overline{x}$, which is a useful property that has many applications in computing reachable sets \cite{CooganFinitieAbstraction,9304391} and  designing interval observers \cite{IntervalLip,IntervalEstimationNonlin,Me,LCSS_paper,khaj}. 

The decomposition function for a mixed monotone function on a given set is in general not unique. In the aforementioned applications of mixed monotonicity it is desirable to find a decomposition function that is tight in the sense that the interval $\left[g(\underline{x},\overline{x}),g(\overline{x},\underline{x})\right]$ is the smallest interval that contains the set $\left\{f(x):\underline{x}\le x\le \overline{x}\right\}$. This leads to tighter over-approximations of reachable sets or interval estimates in interval observers. In \cite[Theorem 1]{tight_decomp}, it is shown that a tight decomposition function can be implicitly constructed from an optimization problem. For example, if $f$ is a univariate scalar function with well defined extrema, then, according to \cite[Theorem 1]{tight_decomp}, the following is a tight mixed monotone decomposition function:
\begin{align}
g(x,y) = \text{opt}_\xi^{(x,y)}f(\xi),\label{e:tight}
\end{align}
where 
\begin{align*}
\text{opt}_\xi^{(x,y)}f(\xi) = \left\{\begin{array}{ll} \inf_{\xi\in[x,y]}f(\xi) & x\le y \\  \sup_{\xi\in[y,x]}f(\xi) & x> y \end{array}\right..
\end{align*}
Generally,  \eqref{e:tight} cannot be evaluated efficiently. To use the property \eqref{e:boundingproperty} in applications, one must be able to evaluate the decomposition function $g$ efficiently. Evaluable decomposition functions come usually at the expense of tightness. An evaluable decomposition functions is typically derived by exploiting some property of the function. For example, a popular way to construct decomposition functions for differentiable functions is to use the bounds of the Jacobian such as the following result: 

\begin{proposition}\label{prop1}
Suppose $f:\mathcal{X}\to\mathcal{T}$ is differentiable everywhere in $\mathcal{X}$ where $\mathcal{X}\subseteq \mathbb{R}$ and $\mathcal{T}\subseteq \mathbb{R}$ and $\|f'(x)\|\le L$ for all $x\in\mathcal{X}$. Then $g(x,y)=f(x)+L(x-y)$ is a mixed monotone decomposition function for $f$ on $\mathcal{X}$. 
\end{proposition}
\begin{proof}
This is a particular case of \cite[Theorem 2]{SuffMixMonotone}.
\end{proof}

To close this section, some other methods for constructing decomposition functions are presented which will be useful in proving the results of this paper.

\subsection{Constructing decomposition functions from other decomposition functions}
It will be useful throughout this paper to construct decomposition functions by considering a function to be composed from a combination of functions which are mixed monotone and have known decomposition functions. 

\subsubsection{Linear combinations of mixed monotone functions} 

One of the most useful results is the following, which guarantees that linear combinations of mixed monotone functions are mixed monotone. A decomposition function can be constructed from the decomposition functions of the functions that are linearly combined. 

\begin{lemma}\label{l:linearcombo}
Let $f_1, \dots, f_m:\mathcal{X}\to\mathcal{T}$ be mixed monotone functions with decomposition functions $g_1,\dots, g_m:\mathcal{X}\times \mathcal{X}\to\mathcal{T}$. Then, $f(x)=\sum_{i=1}^m a_i f_i(x)$ where $a_1,\dots, a_m\in\mathbb{R}$ is mixed monotone over $\mathcal{X}$ and 
\begin{align}
g(x,y) = \sum_{i=1}^m\max(a_i,0)g_i(x,y)+\min(a_i,0)g_i(y,x)\label{e:linearcombodecomp}
\end{align}
is a decomposition function for it. 
\end{lemma}
\begin{proof}
Without loss of generality, let $a_1,\dots, a_m$ be arranged such that $a_1, \dots a_{l-1}\ge 0$ and $a_l, \dots, a_m<0$ where $1\le l \le m$.  With this arrangement, \eqref{e:linearcombodecomp} is equivalent to the following:
\begin{align}
g(x,y) = \sum_{i=1}^{l-1}a_ig_i(x,y)+\sum_{j=l}^ma_jg_j(y,x). 
\end{align}

Now it is demonstrated that $g(x,y)$  satisfies the three properties of a mixed monotone decomposition function in Definition \ref{d:mm}:
\begin{enumerate}
\item Using substitution, $g(x,x)=\sum_{i=1}^ma_ig_i(x,x) = \sum_{i=1}^ma_if_i(x)=f(x)$. So, $f(x)$ is ``embedded'' along the diagonal of $g$.
\item Now consider $x_1\ge x_2$. Since $g_i$ is a decomposition function $g_i(x_1,y)\ge g_i(x_2,y)$ for all $i=1,\dots m$. Therefore, $a_ig_i(x_1,y)\ge a_ig_i(x_2,y)$ for all $i=1,\dots l-1$. It, thus, follows that 
\begin{align}
\sum_{i=1}^{l-1} a_ig_i(x_1,y)\ge \sum_{i=1}^{l-1} a_ig_i(x_2,y). \label{e:sumposcoeff}
\end{align}
Similarly,  $g_j(y,x_1)\le g_j(y,x_2)$ for all $j=1,\dots m$. So, $a_jg_j(y,x_1)\ge a_jg_j(y,x_2)$ for $j=l,\dots m$. Hence, 
\begin{align}
\sum_{j=l}^{m} a_jg_j(y,x_1)\ge \sum_{j=l}^{m} a_jg_j(y,x_2). \label{e:sumnegcoeff}
\end{align}
Combining \eqref{e:sumposcoeff} and \eqref{e:sumnegcoeff}, yields $g(x_1,y)\ge g(x_2,y)$, which implies that $g$ is increasing in its first argument. 
\item The third properties follows from similar logic to the second property. 
\end{enumerate}
This concludes the proof.
\end{proof}

\subsubsection{Compositions of mixed monotone functions}

The following is a novel result, which is that compositions of mixed monotone functions are mixed monotone, and a decomposition function can be constructed using some compositions and argument swapping. 
\begin{lemma}\label{l:composition}
Let $f_1:\mathcal{X}_1\to\mathcal{T}_1$ and $f_2:\mathcal{X}_2\to \mathcal{T}_2$ where $ \mathcal{T}_1 \subseteq \mathcal{X}_2$ be mixed monotone functions with decomposition functions $g_1$ and $g_2$, respectively. The composition $f_{1\circ 2}=f_1\circ f_2:\mathcal{X}_1\to\mathcal{T}_2$ is mixed monotone over $\mathcal{X}_1$ and
\begin{align}
g_{1\circ 2}(x,y) = g_1(g_2(x,y),g_2(y,x))\label{e:compmmd}
\end{align}
is a decomposition function for it.
\end{lemma}
\begin{proof}
Now it is demonstrated that $g_{1\circ 2}(x,y)$ satisfies the three properties of a mixed monotone decomposition function in Definition \ref{d:mm}:
\begin{enumerate}
    \item By substitution,  $g_{1\circ 2}(x,x) = g_1(g_2(x,x),g_2(x,x)) = g_1(f_2(x),f_2(x))=f_1\circ f_2(x)$ for all $x \in\mathcal{X}_1$.
    \item Consider $x,x_1,x_2,y\in\mathcal{X}_1$ where $x_1\ge x_2$. Since $g_2$ is a decomposition function $g_2(x_1,y)\ge g_2(x_2,y)$. Since $g_1$ is also decomposition function,  
    \begin{align}
        g_1(g_2(x_1,y),g_2(y,x))\ge g_1(g_2(x_2,y),g_2(y,x)).\label{e:intermediate}
    \end{align}
    Similarly, $g_2(y,x_1))\le g_2(y,x_2)$, so 
    \begin{align}
        g_{1\circ 2}(x_1,y)&=g_1(g_2(x_1,y),g_2(y,x_1)) \nonumber\\
        &\ge g_1(g_2(x_1,y),g_2(y,x_2)).\label{e:intermediate2}
    \end{align}
    Substituting $x_2$ for $x$ in \eqref{e:intermediate} yields:
    \begin{align}
        g_1(g_2(x_1,y),g_2(y,x_2))\ge &g_1(g_2(x_2,y),g_2(y,x_2)) \nonumber\\
         &=g_{1\circ 2}(x_2,y).\label{e:intermediate3}
    \end{align}
    Combining \eqref{e:intermediate2} and \eqref{e:intermediate3} yields $g_{1\circ 2}(x_1,y)\ge g_{1\circ 2}(x_2,y)$. So $g_{1\circ 2}$ is increasing in its first argument.
    \item Consider $x,y,y_1,y_2\in\mathcal{X}_1$ where $y_1\ge y_2$, so $g_2(x,y_1)\le g_2(x,y_2)$. Therefore, 
    \begin{align}
     g_1(g_2(x,y_2),g_2(y,x))\ge g_1(g_2(x,y_1),g_2(y,x)).   \label{intermediate4}
    \end{align}
    Furthermore, $g_2(y_2,x)\le g_2(y_1,x)$, so 
    \begin{align}
        g_1(g_2(x,y_1),g_2(y_2,x))\ge & g_1(g_2(x,y_1),g_2(y_1,x))\nonumber\\
        & = g_{1\circ 2}(x,y_1).\label{intermediate5}
    \end{align}
    Substituting $y_2$ for $y$ in \eqref{intermediate4} yields
    \begin{align}
    g_{1\circ 2}(x,y_2) = &g_1(g_2(x,y_2),g_2(y_2,x))\nonumber\\
    &\ge g_1(g_2(x,y_1),g_2(y_2,x)).   \label{intermediate6}
    \end{align}
    Combining \eqref{intermediate5} and \eqref{intermediate6} yields $g_{1\circ 2}(x,y_2)\ge g_{1\circ 2}(x,y_1)$, which means that $g_{1\circ 2}$ is decreasing in its second argument. 
\end{enumerate}
This concludes the proof.
\end{proof}

\section{Mixed Monotonicity of Univariate Polynomials}
The following theorem is the main result of this paper. This theorem guarantees that every univariate scalar polynomial function is mixed monotone on $\mathbb{R}$ and the decomposition function is the difference of two monotonically increasing polynomials. This result exploits the fact that the derivative of a polynomial $p(x)$ can be expressed using the Gram matrix representation:
\begin{align}
p'(x) = {x^{\{\sigma\}}}^\top\left(G+L(\alpha)\right)x^{\{\sigma\}},\label{e:der}
\end{align}
 where $\sigma$ is defined according to the degree of $p'(x)$.

\begin{theorem}\label{t:mmpoly}
Let $p(x):\mathbb{R}\to\mathbb{R}$ be a univariate polynomial and let its derivative be denoted using the Gram matrix representation shown in \eqref{e:der}. There exist polynomials $q(x):\mathbb{R}\to\mathbb{R}$ and $r(y):\mathbb{R}\to\mathbb{R}$ defined by
 \begin{subequations}\label{e:pmono}
\begin{align}
q(x) &= \int {x^{\{\sigma\}}}^\top\mathcal{U}{x^{\{\sigma\}}} dx+p(0),\label{e:pup}\\
r(y)&=\int {y^{\{\sigma\}}}^\top\mathcal{V}{y^{\{\sigma\}}}dy, \label{e:pdown}
\end{align}
\end{subequations}
where  $\mathcal{U}$ and $\mathcal{V}$ are computed from the Gram matrix for $p'(x)$ as follows:
 \begin{subequations}\label{e:gramstuff}
 \begin{align}
 G+L(\alpha) =\mathcal{U}-\mathcal{V},\label{e:decompmatrix}\\
\mathcal{U}, \mathcal{V}\succeq 0,
 \end{align} 
 \end{subequations}
 such that 
\begin{align}
g(x,y)=q(x)-r(y)\label{e:decomp}
\end{align}
 is a mixed monotone decomposition function for $p$ on $\mathbb{R}$. 
\end{theorem}
\begin{proof}
It follows from Lemma \ref{emma} that for any $\alpha$ there exist $\mathcal{U}$ and $\mathcal{V}$ such that \eqref{e:gramstuff} holds. Therefore, $\alpha$ can be chosen arbitrarily. Once  $\alpha$, $\mathcal{U}$, and $\mathcal{V}$ are chosen such that \eqref{e:gramstuff} holds, the resulting decomposition function  \eqref{e:decomp} satisfies the three properties in Definition \ref{d:mm}: 
\begin{enumerate}
\item By the definition of the (anti)derivative, $p(x)=\int p'(x)dx+p(0)$. It then follows from \eqref{e:der}-\eqref{e:decompmatrix} and \eqref{e:decomp} that 
\begin{align*}
 p(x)&= \int \left({x^{\{\sigma\}}}^\top\left(G+L(\alpha)\right){x^{\{\sigma\}}}\right)dx+p(0)\\
  &=\int \left({x^{\{\sigma\}}}^\top\mathcal{U}{x^{\{\sigma\}}}-{x^{\{\sigma\}}}^\top\mathcal{V}{x^{\{\sigma\}}}\right)dx +p(0)\\
  &=q(x)-r(x)=g(x,x).
 \end{align*}
 Therefore, $p$ is ``embedded'' along the diagonal of $g$, so the first property holds. 
\item The partial derivative $\frac{\partial g}{\partial x}(x,y) = {x^{\{\sigma\}}}^\top\mathcal{U}{x^{\{\sigma\}}}$. Since $\mathcal{U}\succeq 0$, $\frac{\partial g}{\partial x} (x,y)\ge 0$ for all $(x,y)\in\mathbb{R}\times \mathbb{R}$. Therefore, $g$ is monotonically increasing in its first argument and the second property follows. 
\item The partial derivative $\frac{\partial g}{\partial y}(x,y) = -{y^{\{\sigma\}}}^\top\mathcal{V}{y^{\{\sigma\}}}$. Since $\mathcal{V}\succeq 0$, $\frac{\partial g}{\partial y} (x,y)\le 0$ for all $(x,y)\in\mathbb{R}\times \mathbb{R}$. Therefore, $g$ is monotonically decreasing in its second argument and the third property follows. 
\end{enumerate}
Therefore, $g(x,y)$ defined in \eqref{e:decomp}  is a mixed monotone decomposition function for $p(x)$ on $\mathbb{R}$. Since  $q(x)$ and $r(y)$ are defined in \eqref{e:pup} and \eqref{e:pdown}, respectively, as integrals of polynomials, they are also polynomials. 
\end{proof}

Theorem \ref{t:mmpoly}  provides a straightforward procedure to compute a global decomposition function for any univariate polynomial: once the Gram matrix for $p'(x)$  is computed \eqref{e:der} it is decomposed into a difference of two positive semidefinite matrices \eqref{e:gramstuff} to then compute \eqref{e:pmono} which leads to \eqref{e:decomp}.

 Before discussing Theorem \ref{t:mmpoly}  further, the following corollary to Theorem \ref{t:mmpoly}  is presented which shows a sufficient condition to check if a given polynomial is monotonically increasing or decreasing.  The proof is straightforward, so it is omitted. 
 \begin{corollary}
 Let $p(x):\mathbb{R}\to\mathbb{R}$ be a polynomial with a derivative that has a Gram matrix representation as shown in \eqref{e:der}. If there exists $\alpha$ such that $ G+L(\alpha)\succeq 0$ (resp. $\preceq 0$), then $p(x)$ is monotonically increasing (resp. decreasing) on $\mathbb{R}$. 
 \end{corollary}

 \subsection{Polynomial decomposition function tightness}

For a chosen $\alpha$, the matrix decomposition in \eqref{e:gramstuff} is not unique (see Remark \ref{r:uniqueness}). Moreover, different choices of $\alpha$ in the Gram matrix for $p'(x)$ will lead to different matrix decompositions in \eqref{e:gramstuff}.  Therefore, different choices of  $\alpha$, $\mathcal{U}$, and $\mathcal{V}$ will yield different polynomial decomposition functions for $p(x)$.

It is desirable to choose $\alpha$, $\mathcal{U}$, and $\mathcal{V}$ to yield a decomposition function that is as tight as possible. That is, find $\alpha$, $\mathcal{U}$, and $\mathcal{V}$ such that $\left\|g(\overline{x},\underline{x})- g(\underline{x},\overline{x})\right\|$ is small. One way to choose $\alpha$, $\mathcal{U}$, and $\mathcal{V}$ is to solve a semidefinite programming problem with the following form:

   \begin{equation}
\begin{aligned}
& \underset{\alpha, \mathcal{U}, \mathcal{V}}{\text{minimize}} 
& & J(\alpha, \mathcal{U}, \mathcal{V})\\
& \text{subject to}
& &  \eqref{e:gramstuff}
\end{aligned}
\label{e:sdp}
\end{equation}

One of the main questions is then to specify which objective function $J(\alpha, \mathcal{U}, \mathcal{V})$ should be minimized to construct the tightest decomposition function possible. A reasonable objective function the following:
\begin{align}
J(\alpha, \mathcal{U}, \mathcal{V})=\|\mathcal{U}\|_F+\|\mathcal{V}\|_F.\label{e:frobob}
\end{align}
The reasoning for this is as follows: by the triangular inequality,  
\begin{align*}
\left\|g(\overline{x},\underline{x})- g(\underline{x},\overline{x})\right\|\le \left\|q(\overline{x})-q(\underline{x})\right\| +\left\|r(\overline{x})-r(\underline{x})\right\|.
\end{align*}
By the mean value theorem \cite[Theorem 5.10]{rudin}, 
\begin{align*}
q(\overline{x})-q(\underline{x})=q'(\xi_1)(\overline{x}-\underline{x})={\xi_1^{\{\sigma\}}}^\top\mathcal{U}{\xi_1^{\{\sigma\}}}(\overline{x}-\underline{x}),
\end{align*}
and
\begin{align*}
r(\overline{x})-r(\underline{x})=r'(\xi_2)(\overline{x}-\underline{x})={\xi_2^{\{\sigma\}}}^\top\mathcal{V}{\xi_2^{\{\sigma\}}}(\overline{x}-\underline{x}).
\end{align*}
 for some $\xi_1$ and $\xi_2$ such that  $\underline{x}\le\xi_1,\xi_2\le\overline{x}$. It then follows that 
 \begin{align*}
  \left\|g(\overline{x},\underline{x})- g(\underline{x},\overline{x})\right\|\le \|\mathcal{U}\|_F\left\|{\xi_1^{\{\sigma\}}}\right\|^2\|\overline{x}-\underline{x}\|\\
  +\|\mathcal{V}\|_F\left\|{\xi_2^{\{\sigma\}}}\right\|^2\|\overline{x}-\underline{x}\|.
 \end{align*}
 This suggests that making $\|\mathcal{U}\|_F$ and $\|\mathcal{V}\|_F$ smaller tends to lead to tighter decomposition functions. 
 
 Later on, in \S\ref{s:objective}, it  will be  shown via example that different decomposition functions derived from different objective functions can have a noticeable effect on tightness of the decomposition function. From the authors experience, using the objective function \eqref{e:frobob} tends to yield tighter decomposition functions. It could be the case that other objective functions yield tighter decomposition functions for different polynomials. Therefore, in practice, one must experiment with different objective functions and check tightness in the regimes of interest. 

\section{Mixed Monotonicity of Multivariate Polynomials}

Armed with Theorem \ref{t:mmpoly} and Lemmas \ref{l:linearcombo} and \ref{l:composition}, it can now be proven that products of mixed monotone functions are mixed monotone. This will be useful in proving that all multivariate polynomial functions are mixed monotone and have polynomial decomposition functions.
\begin{lemma}\label{l:product}
Let $f_1:\mathcal{X}\to \mathcal{T}$ and $f_2:\mathcal{X}\to\mathcal{T}$ where $\mathcal{T}\subseteq \mathbb{R}$ be mixed monotone functions with decomposition functions $g_1$ and $g_2$, respectively. The product $f_{1\times 2} (x)= f_1(x)f_2(x)$ is a mixed monotone function over $\mathcal{X}$ and 
\begin{align}
    &g_{1\times 2}(x,y) = \frac{1}{2} g_{sq}(g_1(x,y)+g_2(x,y),g_1(y,x)+g_2(y,x))\nonumber\\
    &-\frac{1}{2}g_{sq}(-g_1(x,y)+g_2(y,x),-g_1(y,x)+g_2(x,y))\label{e:mmproduct}
\end{align}
is a decomposition function for it, where $f_{sq}:\mathbb{R}\to\mathbb{R}$ is the squaring function $f_{sq}(z) = z^2$ and $g_{sq}$ is any decomposition function for it.
\end{lemma}
\begin{proof}
The first step of the proof is to express $f_{1\times 2}(x)$ in a quadratic form as follows:
\begin{align*}
f_{1\times 2}(x) = \begin{bmatrix}f_1(x)\\ f_2(x) \end{bmatrix}^\top \begin{bmatrix} 0 & \frac{1}{2}\\ \frac{1}{2} & 0\end{bmatrix}\begin{bmatrix}f_1(x)\\ f_2(x) \end{bmatrix}.
\end{align*}
Using an eigendecomposition, this can be further expanded to the following:
\begin{align}
&f_{1\times 2}(x)= \frac{1}{2}\left((f_1(x)+f_2(x))^2-(-f_1(x)+f_2(x))^2\right)\nonumber\\
&=\frac{1}{2} \left( f_{sq}\circ(f_1+f_2)(x)-f_{sq}\circ(-f_1+f_2)(x) \right).  \label{e:expanded}
\end{align}

By Theorem \ref{t:mmpoly}, $f_{sq}$ is mixed monotone on $\mathbb{R}$. Hence, in \eqref{e:expanded}, $f_{1\times 2}$ is expressed as a linear combination of compositions of mixed monotone functions. By  Lemmas \ref{l:linearcombo} and \ref{l:composition}, this means that $f_{1\times 2}$ is mixed monotone. Since $f_{sq}$ is mixed monotone, a decomposition function $g_{sq}$ exists. The decomposition \eqref{e:mmproduct} can then be constructed using Lemmas \ref{l:linearcombo}-\ref{l:composition}.
\end{proof}

\begin{remark}\label{r:products}
Lemma \ref{l:product} can be used to show that products of more than two functions are mixed monotone by induction. Suppose $f_1, \dots f_m$ are mixed monotone functions. The function $f(x)=\prod _{i=1}^mf_i(x)$ is mixed monotone and this follows because, by Lemma \ref{l:product}, $f_{1\times 2} (x)= f_1(x)f_2(x)$ is mixed monotone and has decomposition function $g_{1\times 2}$ shown in \eqref{e:mmproduct}. Then, using Lemma \ref{l:product} again,  $f_{1\times 2\times 3}(x)= f_1(x)f_2(x)f_3(x)$ is mixed monotone and has the following decomposition function: $g_{1\times 2\times 3}(x,y) = \frac{1}{2} g_{sq}(g_{1\times 2}(x,y)+g_3(x,y),g_{1\times 2}(y,x)+g_3(y,x))-\frac{1}{2}g_{sq}(-g_{1\times 2}(x,y)+g_3(y,x),-g_{1\times 2}(y,x)+g_3(x,y))$. This process is then repeated $m-3$ more times to construct a decomposition function for $f$.
\end{remark}

The following theorem shows that all multivariate monomials are mixed monotone and have polynomial decomposition functions.The idea behind the proof is to use Theorem \ref{t:mmpoly} and Lemma \ref{l:product} to prove that monomials of multivariate polynomials are mixed monotone. It then follows from Lemma \ref{l:linearcombo} that every polynomial is mixed monotone.

\begin{theorem}\label{t:mmmultivariate}
Every polynomial $p:\mathbb{R}^n\to \mathbb{R}$ is mixed monotone on $\mathbb{R}^n$ and has a polynomial decomposition function. 
\end{theorem}
\begin{proof}
To prove this result, it is sufficient to prove that $x^\alpha$ is a mixed monotone function for any $n$-tuple $\alpha = (\alpha,\dots, \alpha_n)$, $\alpha_i\in\mathbb{N}_0$,  and has a polynomial decomposition function. Then the result follows using Lemma \ref{l:linearcombo}.

Let $p_d :\mathbb{R}\to\mathbb{R}$ be $p_d(z)=z^d$ where $d\in\mathbb{N}_0$. By Theorem \ref{t:mmpoly}, $p_d(z)$ is mixed monotone over $\mathbb{R}$ and has a polynomial decomposition function. Therefore, using Lemma \ref{l:composition}, $p_d(e_i^\top x) = p_d(x_i)$ is mixed monotone over $\mathbb{R}^n$ and has a polynomial decomposition function $g_i(x,y)$.

By Lemma \ref{l:product}, $f_{i\times j}(x) = p_{\alpha_i}(e_i^\top x)p_{\alpha_j}(e_j^\top x) = x_i^{\alpha_i}x_j^{\alpha_j}$ is mixed monotone over $\mathbb{R}^n$, where $i,j\in \{1,\dots, n\}$ and $g_{i\times j}(x,y) = \frac{1}{2} g_{sq}(g_i(x,y)+g_j(x,y),g_i(y,x)+g_j(y,x))-\frac{1}{2}g_{sq}(-g_i(x,y)+g_j(y,x),-g_i(y,x)+g_j(x,y))$ is a decomposition function for it. If $g_{sq}$ is chosen to be a polynomial decomposition function for $f_{sq}$, which is always possible due to Theorem \ref{t:mmpoly}, and $g_i, g_j$ are also chosen to be polynomials, then $g_{i\times j}(x,y)$ is a polynomial in $(x,y)$ since it is comprised of compositions and linear combinations of polynomials. 

By induction (cf. Remark \ref{r:products}), it follows that $x^\alpha = \prod_{i=1}^nf_{\alpha_i}(e_i^\top x)$ is a mixed monotone function on $\mathbb{R}^n$ and has a polynomial decomposition function. 
\end{proof}
 \section{Examples}
 In this section five examples are provided. In the examples,  \verb|Convex.jl|\footnote{https://jump.dev/Convex.jl/stable/} with the \verb|SCS| solver  was used to solve the semidefinite program \eqref{e:sdp}. To verify nonnegativity of polynomials using sum of squares techniques, the package \verb|SumOfSquares.jl|\footnote{https://jump.dev/SumOfSquares.jl/stable/} was used. 
 \subsection{Example 1: Comparison to decomposition function based on Jacobian bounds }
Consider the following polynomial as an example:
\begin{align}
p_1(x) = x^2+1. \label{e:example}
\end{align}
Its derivative has the following Gram matrix representation:
\begin{align*}
p_1'(x) = \begin{bmatrix} 1 \\ x \end{bmatrix}^\top \begin{bmatrix}0 & 1\\ 1 & 0\end{bmatrix} \begin{bmatrix} 1 \\ x \end{bmatrix}.
\end{align*}
 Note that for $x^{\{1\}}$, $\mathcal{L}$ defined in \eqref{e:Lset} is a singleton with the zero matrix so there is no $\alpha$ to solve for in this example. The semidefinite program \eqref{e:sdp} with the objective function \eqref{e:frobob} was solved to find 
\begin{align*}
\mathcal{U} = \frac{1}{2}\begin{bmatrix}1 & 1\\ 1& 1\end{bmatrix},\; \mathcal{V} = \frac{1}{2}\begin{bmatrix}1 & -1\\ -1& 1\end{bmatrix}.
\end{align*}
From \eqref{e:pmono}-\eqref{e:decomp}, the following decomposition function is constructed:
\begin{align}
g_1(x,y) = \frac{x^3}{6}+\frac{x^2}{2}+\frac{x}{2}-\frac{y^3}{6}+\frac{y^2}{2}-\frac{y}{2}+1.\label{decompex}
\end{align}

This decomposition satisfies the three properties shown in Definition \ref{d:mm}:
\begin{enumerate}
\item The first property can be verified by substitution to find $g_1(x,x)=p_1(x)$. 
\item The partial derivative $\frac{\partial g_1}{\partial x}(x,y) = \frac{x^2}{2}+x+\frac{1}{2}= (\frac{\sqrt{2}}{2}x+\frac{\sqrt{2}}{2})^2\ge 0$ for all $(x,y)\in\mathbb{R}\times\mathbb{R}$, so $g_1$ is monotonically increasing in its first argument. 
\item The partial derivative $\frac{\partial g_1}{\partial y}(x,y) = -(\frac{\sqrt{2}}{2}-\frac{\sqrt{2}}{2}y)^2\le 0$ for all $(x,y)\in\mathbb{R}\times\mathbb{R}$, so $g_1$ is monotonically decreasing in its second argument. 
\end{enumerate}

To illustrate the decomposition function, consider Fig. \ref{f:mm} where for $z=-5$ to $z=5$, $g_1(z+1,z-0.5)$ is plotted in red and $g_1(z-0.5,z+1)$ is plotted in blue. Additionally, $p_1(z+1), p_1(z+0.5), p_1(z),p_1(z-0.25),$ and $p_1(z-0.5)$ are plotted in gray. Since $g_1$ is a decomposition function, then $g_1(z-0.5,z+1)\le p_1(z+1), p_1(z+0.5), p_1(z),p_1(z-0.25),p_1(z-0.5)\le g_1(z+1,z-0.5)$ for all $z\in\mathbb{R}$. This can clearly be seen in Fig. \ref{f:mm} by the fact that all of the gray lines lie in between the red and blue lines. It appears that the decomposition function is less tight as $\|z\|$ increases.

\begin{figure}[!t]
\centerline{\includegraphics[width=\columnwidth]{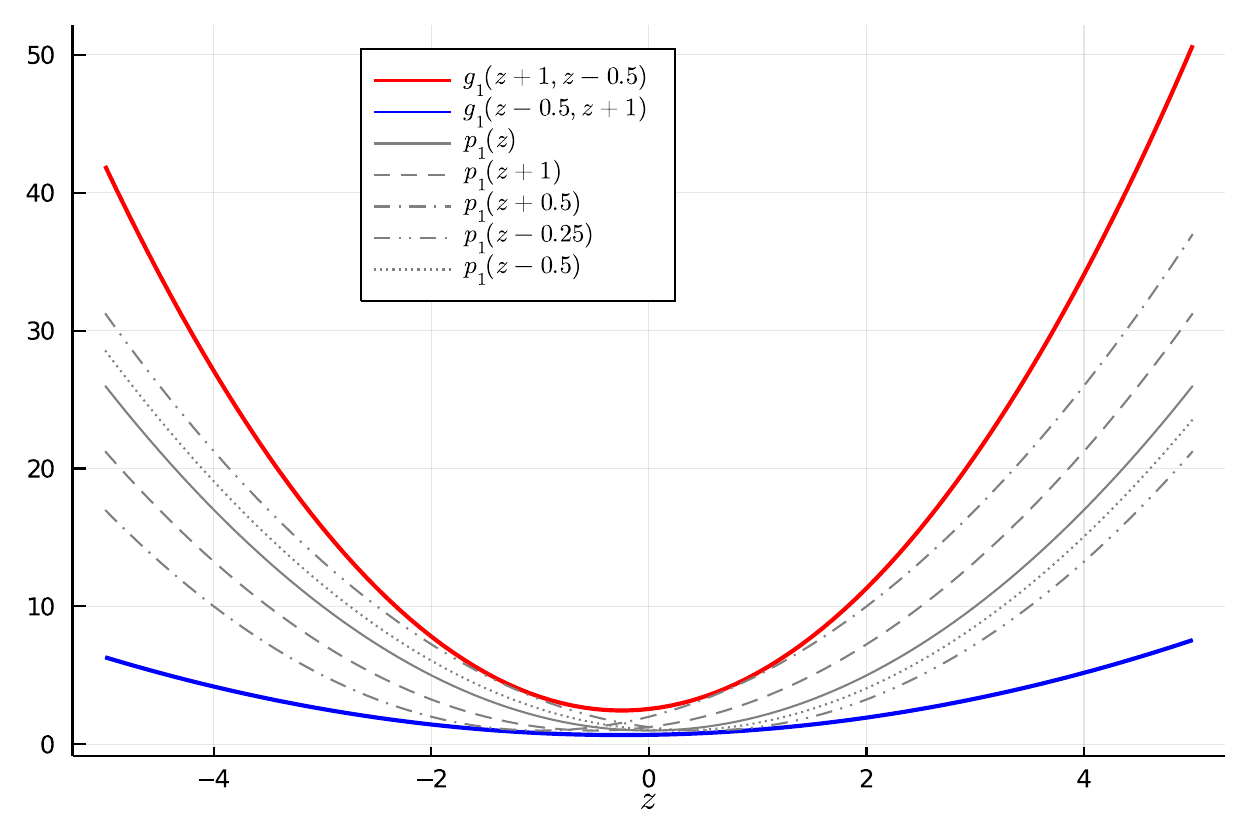}}
\caption{Depiction of the decomposition function \eqref{decompex} for \eqref{e:example}.}
\label{f:mm}
\end{figure}

The following is a mixed monotone decomposition function for $p_1(x)$ on  $\mathcal{X}=\{x:\|x\|\le 2\}$:
\begin{align}
g_2(x,y) = p_1(x)+4(x-y),\label{e:linmmp}
\end{align}
which was constructed from Proposition \ref{prop1} using Jacobian bounds on $\mathcal{X}$. 

\begin{figure}[!t]
\centerline{\includegraphics[width=\columnwidth]{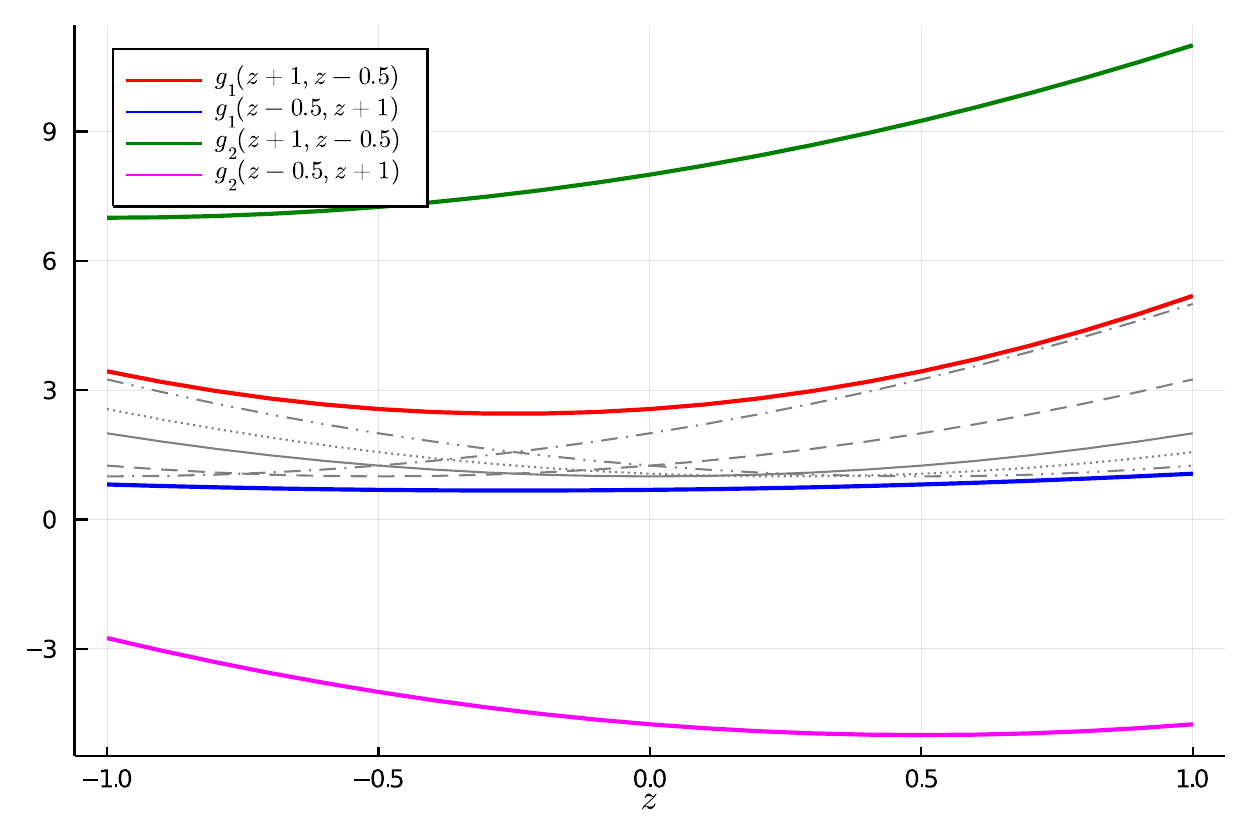}}
\caption{Comparison of the polynomial decomposition function given in \eqref{decompex} and the decomposition function derived from Jacobian bounds on $\mathcal{X}=\{x:\|x\|\le 2\}$ given in \eqref{e:linmmp}. The gray lines are the same as shown in Fig. \ref{f:mm}.}
\label{f:mm2}
\end{figure}

In Fig. \ref{f:mm2}, for $z=-1$ to $z=1$, $g_2(z+1,z-0.5)$ is plotted in green and $g_2(z-0.5,z+1)$ is plotted in magenta along with $g_1(z+1,z-0.5)$ in red and $g_1(z-0.5,z+1)$ in blue. The gray lines are the same as shown in Fig. \ref{f:mm}.  In addition to being a global decomposition function, it is clear from Fig. \ref{f:mm2} that the polynomial decomposition function is much tighter than the decomposition function derived from Jacobian bounds. 

 \subsection{Example 2: Tightness of polynomial decomposition functions derived from different objective functions}\label{s:objective}
 Consider the fourth Legendre polynomial as an example:
 \begin{align*}
 p_2(x) = \frac{35}{8}x^4-\frac{15}{4}x^2+\frac{3}{8}.
 \end{align*}
Its derivative has the following Gram matrix representation:
\begin{align*}
p_2'(x) = \begin{bmatrix} 1 \\ x \\ x^2 \end{bmatrix}^\top \begin{bmatrix}0 & -\frac{15}{4}& \alpha \\ -\frac{15}{4} & -2\alpha & \frac{35}{4}&\\ \alpha & \frac{35}{4} & 0 \end{bmatrix} \begin{bmatrix} 1 \\ x \\ x^2\end{bmatrix}.
\end{align*}

The semidefinite program \eqref{e:sdp} with the objective function \eqref{e:frobob} was solved which leads to the following decomposition function using \eqref{e:pmono}-\eqref{e:decomp}:
\begin{align}
g_3(x,y) &= 0.375 + 0.9206 x - 1.875x^{2} + 0.4897x^{3} \nonumber\\
&+ 2.1875 x^{4} + 0.8109x^{5}- 0.9206 y - 1.875y^{2} \nonumber\\
&- 0.4897 y^{3} + 2.1875y^{4} - 0.8109 y^{5}.\label{e:frob}
\end{align}

Solving \eqref{e:sdp} with the following objective function:
\begin{equation}
J(\alpha, \mathcal{U}, \mathcal{V})=\|\mathcal{U}\|_1+\|\mathcal{V}\|_1\label{e:1normob}
\end{equation}
 leads to the following decomposition function using \eqref{e:pmono}-\eqref{e:decomp}:
\begin{align}
g_4(x,y) &= 0.375 + 8.1767x - 1.875 x^{2} + 1.2672x^{3} \nonumber\\
&+ 2.1875 x^{4} + 1.1353x^{5} - 8.1767y - 1.875y^{2} \nonumber\\
&- 1.2672 y^{3} + 2.1875 y^{4} - 1.1353 y^{5}.\label{e:l1}
\end{align}
\begin{figure}[!t]
\centerline{\includegraphics[width=\columnwidth]{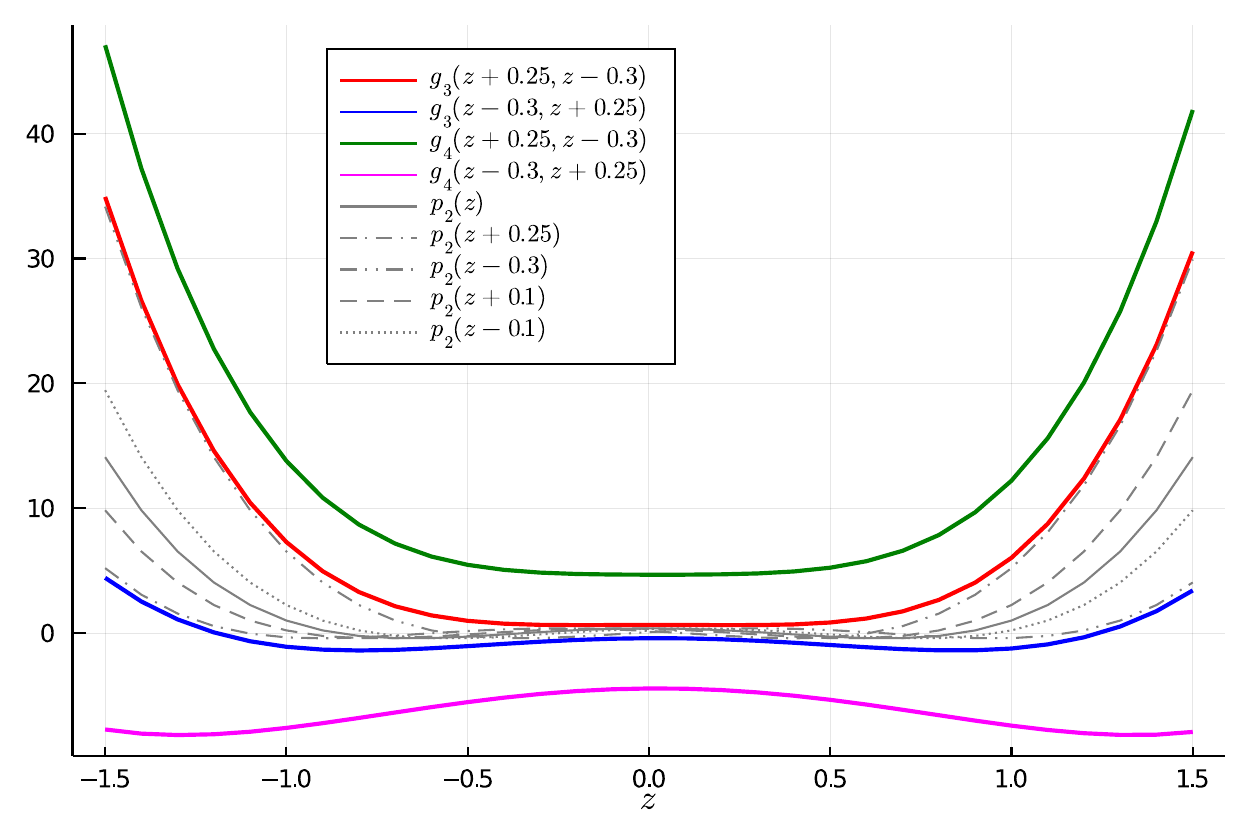}}
\caption{Comparison of the decomposition function \eqref{e:frob} derived by minimizing  \eqref{e:frobob} and the decomposition function \eqref{e:l1} derived by minimizing \eqref{e:1normob}.}
\label{f:mml}
\end{figure}
Note that solving \eqref{e:sdp} yielded $\alpha=0$ for both objective functions. 

Fig. \ref{f:mml} shows a comparison of the decomposition function $g_3$ given in \eqref{e:frob} and $g_4$ given in \eqref{e:l1} by plotting $g_3(z+0.25,z-0.3)$, $g_3(z-0.3,z+0.25)$, $g_4(z+0.25,z-0.3)$ and $g_4(z-0.3,z+0.25)$ for $z=-1.5$ to $z=1.5$ in red, blue, green, and magenta, respectively. It can be clearly seen that the interval $[g_3(z-0.3,z+0.25),g_3(z+0.25,z-0.3)]$ is tighter than the interval $[g_4(z-0.3,z+0.25),g_4(z+0.25,z-0.3)]$ for all $z=-1.5$ to $z=1.5$.

\subsection{Example 3: Checking monotonicity of a polynomial}
Consider the following polynomial:
\begin{align}
p_3(x) = \frac{x^5}{5}-2x^2+3x.\label{monotone}
\end{align}
Upon inspecting the plot of $p_3$ in Fig. \ref{f:monotone}, it appears that $p_3$ is monotonically increasing. Here, Corollary 1 is used to verify that it is indeed monotonically increasing.

\begin{figure}[!t]
\centerline{\includegraphics[width=\columnwidth]{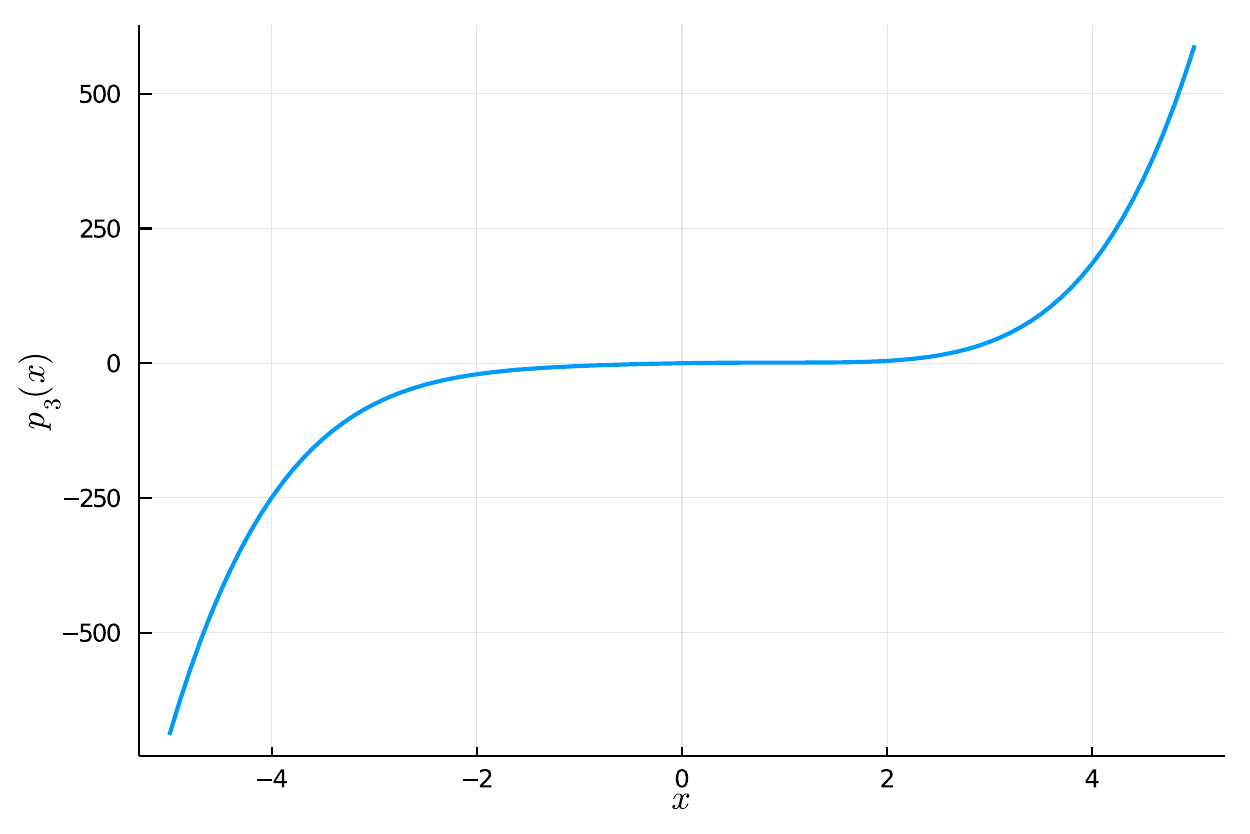}}
\caption{Plot of the polynomial $p_3$ given in \eqref{monotone}.}
\label{f:monotone}
\end{figure}

The derivative of $p_3$ has the following Gram matrix representation:
\begin{align*}
p_3'(x) = &  {x^{\{2\}}}^\top\left(G+L(\alpha)\right)x^{\{2\}}\\
= &\begin{bmatrix} 1 \\ x \\ x^2 \end{bmatrix}^\top \begin{bmatrix} 3 & -2  & -\alpha \\ -2 & 2\alpha & 0\\ -\alpha & 0 & 1\end{bmatrix} \begin{bmatrix} 1 \\ x \\ x^2\end{bmatrix}.
\end{align*}
The matrix $G+L(1) \succeq 0$. Hence, by Corollary 1, $p_3$ given in \eqref{monotone} is monotonically increasing on $\mathbb{R}$.

\subsection{Example 4: Multivariate polynomial function}
Consider the following multivariate monomial:
\begin{align}
p_m(x) = x_1x_2.
\end{align}
According to Theorem \ref{t:mmmultivariate}, $p_m$ is mixed monotone globally over $\mathbb{R}^2$ and has a polynomial decomposition function. The monomial $p_m$ is the product of two mixed monotone functions ($p_m(x)=f_1(x)f_2(x)$ where $f_1(x)=e_1^\top x=x_1$ and $f_2(x)=e_2^\top x=x_2$). Therefore, using Lemma \ref{l:product}, $p_m$ has the following mixed monotone decomposition function:
\begin{align}
&g_m(x,y) = -0.3536y_2 - 0.3536y_1 + 0.3536x_2 + 0.3536x_1\nonumber\\
&+ 0.25y_1y_2 + 0.25x_2y_1 + 0.25x_1y_2 + 0.25x_1x_2\nonumber \\
&- 0.05893y_2^3 - 0.08839y_1y_2^2 - 0.08839y_1^2y_2 - 0.05893y_1^3\nonumber\\
& + 0.08839x_2y_1^2 - 0.08839x_2^2y_1 + 0.05893x_2^3  \nonumber\\
&+0.08839x_1y_2^2 + 0.08839x_1x_2^2 - 0.08839x_1^2y_2\nonumber \\
&+ 0.08839x_1^2x_2 + 0.058934x_1^3.\label{e:multivar}
\end{align}
To compute this decomposition function, $g_{sq}(x,y)$ was chosen similarly to Example 1 (i.e. $g_{sq}(x,y)=g_1(x,y)-1$ where $g_1$ is given in  \eqref{decompex}). Now it is shown that $g_m$ satisfies all three properties of a decomposition function for $p_m$.

\begin{enumerate}
\item The first property can be verified by substitution to find $g_m(x,x)=p_m(x)$. 
\item Now it will be shown that $g_m$ is increasing in its first argument by showing that the partial derivative with respect to $x_1$ and $x_2$ can be written as sums of squares:
\begin{enumerate}
\item The partial derivative $\frac{\partial g_m}{\partial x_1}(x,y) = 0.3536 + 0.25y_2 + 0.25x_2 + 0.08839y_2^2 + 0.08839x_2^2 - 0.1768x_1y_2 + 0.1768x_1x_2 + 0.1768x_1^2 = 
(-0.5946- 0.2102y_2 - 0.2102x_2)^2 + (-0.2102y_2 - 0.2102x_2 - 0.4204x_1)^2\ge 0$ for all $(x,y)\in\mathbb{R}^2\times \mathbb{R}^2$.
\item The partial derivative $\frac{\partial g_m}{\partial x_2}(x,y) = (-0.5946- 0.2102y_1 - 0.2102x_1)^2 + ( 0.2102y_1 - 0.4204x_2 - 0.2102x_1)^2\ge 0$ for all $(x,y)\in\mathbb{R}^2\times \mathbb{R}^2$.
\end{enumerate}
\item Finally, $g_m$ is shown to be decreasing in its second argument by showing that the partial derivatives with respect to $y_1$ and $y_2$ can be expressed as the negations of sums of squares:
\begin{enumerate}
\item The partial derivative $\frac{\partial g_m}{\partial y_1}(x,y) =-(-0.5946+ 0.2102y_2 + 0.2102x_2)^2 -(0.2102y_2 + 0.4204y_1 - 0.2102x_1)^2 _2\le 0$ for all $(x,y)\in\mathbb{R}^2\times \mathbb{R}^2$.
\item The partial derivative $\frac{\partial g_m}{\partial y_2}(x,y) = -(-0.5946 + 0.2102y_1 + 0.2102x_1)^2 - (0.4204y_2 - 0.2102y_2 + 0.2102_1x_1)^2\le 0 $  for all $(x,y)\in\mathbb{R}^2\times \mathbb{R}^2$.
\end{enumerate}
\end{enumerate}

 \subsection{Example 5: Reachable set computation}
 Consider the following discrete-time system inspired by the example in \cite[\S VI.B]{10068731}:
 \begin{subequations}\label{e:dt}
 \begin{align}
 x[k+1]=f(x[k])+u[k],\; \forall k=0,1,2,\dots,
 \end{align}
 where 
 \begin{align}
 f(x) = 0.7x+0.32x^2.\label{e:f}
 \end{align}
  \end{subequations}
 Suppose that the input $u[k]$ is bounded as follows:
 \begin{align}
 -0.1\le u[k]\le 0.1,\; \forall k=0,1,2,\dots.\label{e:inputbounds}
 \end{align}
 
  Moreover, suppose that the initial condition belongs to some known interval:
 \begin{align}
 \underline{\xi}_0\le x[0]\le \overline{\xi}_0.\label{e:initbox}
 \end{align}
 
 The exact reachable set after $N$ steps is defined as follows (cf. \cite[Definition 4]{10068731}):
 \begin{align}
 \mathcal{R}_N = \{&x[N]\in \mathbb{R}:  x[k+1]=f(x[k])+u[k],\nonumber\\
 &x[0] \text{ satisfies \eqref{e:initbox}}, u[k] \text{ satisfies \eqref{e:inputbounds}},  \nonumber\\
 & \forall k=0,1,2,\dots,N-1\}.\nonumber
 \end{align}
The exact reachable set $\mathcal{R}_N$ is a difficult set to compute; however, by finding a decomposition function $g$ for $f$ and exploiting the property of decomposition functions \eqref{e:boundingproperty}, an over-approximation of $\mathcal{R}_N$ can be efficiently computed. By propagating the following system:
\begin{subequations}\label{e:overapp}
 \begin{align}
\overline{x}[k+1] &= g(\overline{x}[k],\underline{x}[k])+0.1, \; \overline{x}[0] =  \overline{\xi}_0,\\
\underline{x}[k+1] &= g(\underline{x}[k],\overline{x}[k])-0.1, \; \underline{x}[0] =  \underline{\xi}_0,
 \end{align}
 an over-approximation of the reachable set with a compact interval follows, i.e.,
 \begin{align*}
 \mathcal{R}_N\subseteq \left[\underline{x}[N],\overline{x}[N]\right],\; \forall N=1,2,3,\dots.
 \end{align*}
 The over-approximation is efficiently computed because, once a decomposition function is found, the reachable set at step $N$ is over-approximated by simply evaluating polynomials $2N$ times.

 Solving \eqref{e:sdp} with the objective function \eqref{e:frobob} and using \eqref{e:pmono}-\eqref{e:decomp}, the following decomposition function was found for the polynomial function $f(x)$ in \eqref{e:f}:
 \begin{align}
 g(x,y) &= 0.7163 x + 0.2781 x^{2} + 0.03599 x^{3}\nonumber\\
 &- 0.0163y + 0.0419 y^{2} - 0.03599 y^{3}.
 \end{align}
 \end{subequations}

In Fig. \ref{f:reach}, the over-approximation of the reachable set computed by propagating \eqref{e:overapp} for the system \eqref{e:dt} is shown where $\underline{\xi}_0= x[0]=\overline{\xi}_0=0$ in \eqref{e:initbox}, i.e. reachability from the origin. $\overline{x}[k]$ is shown in red and $\underline{x}[k]$ is shown in blue. Sample trajectories of \eqref{e:dt} with randomly generated inputs that satisfy the bounds \eqref{e:inputbounds} are shown in gray. All of the gray lines are contained within the red and blue lines, i.e. for each trajectory, $x[k]$ belongs the interval $\left[\underline{x}[k],\overline{x}[k]\right]$ for all $k=0,1,2,\dots$. Moreover, it is clear from Fig. \ref{f:reach} that the over-approximation is a tight over-approximation in the sense that the interval $\left[\underline{x}[k],\overline{x}[k]\right]$ is almost the smallest interval that contains all of the gray trajectories.

 \begin{figure}[!t]
\centerline{\includegraphics[width=\columnwidth]{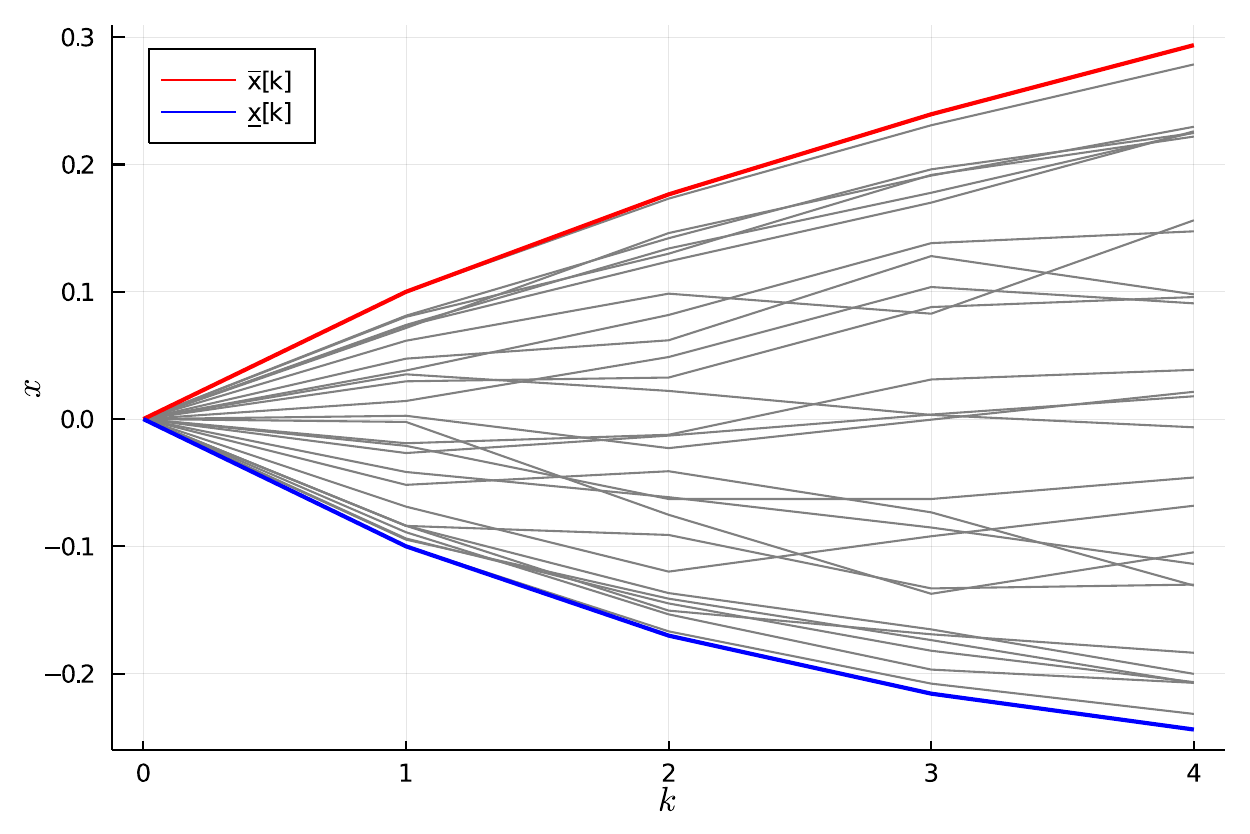}}
\caption{Depiction of the over-approximation of the reachable set from the origin computed by propagating \eqref{e:overapp}. The gray lines are sample trajectories of \eqref{e:dt} with randomly generated inputs that satisfy the bounds \eqref{e:inputbounds}.}
\label{f:reach}
\end{figure}

 \section{Conclusions}
 In this paper, it was shown that every univariate polynomial function is mixed monotone with a decomposition function that is the difference of two monotonically increasing polynomials. Examples of polynomial mixed monotone functions were given. It was shown that not only are the polynomial decomposition functions global decomposition functions, but they can outperform (in terms of tightness) local decomposition functions derived from local Jacobian bounds. Moreover, an example of reachable set computation showed that a tight over-approximation of a discrete-time polynomial system can be efficiently computed using its polynomial decomposition function. 
 
 There are a few avenues to pursue future work. Firstly, an important application would be to consider computing decomposition functions for polynomials with uncertain coefficients. This would be invaluable for constructing interval observers for uncertain polynomial systems. Secondly, further investigation into how to compute the tightest decomposition function should be pursued. More work should be dedicated to constructing decomposition functions for multivariate polynomials.


\bibliographystyle{plain}        
\bibliography{DistBib}           



\appendix

\end{document}